\documentclass{article}
\usepackage{amsmath}
\usepackage[colorlinks=true]{hyperref}
\hypersetup{urlcolor=blue, citecolor=red}
\usepackage{amsmath,amssymb,color}
\usepackage{amsthm}
\usepackage{graphicx}

\newtheorem{theorem}{Theorem}
\newtheorem{lemma}{Lemma}
\newtheorem{corollary}{Corollary}
\newtheorem{remark}{Remark}
\newtheorem{definition}{Definition}
\newtheorem{assumption}{Assumption}

\renewcommand{\d}{\,\text{\rm{}d}}
\newcommand{\dx}{\,\text{\rm{}d}x}
\newcommand{\ds}{\,\text{\rm{}d}s}
\newcommand{\dt}{\,\text{\rm{}d}t}

\newcommand{\be}{\begin{equation}}
\newcommand{\ee}{\end{equation}}

\renewcommand{\phi}{\varphi}

\title{On the switching behavior of sparse optimal  controls for the one-dimensional
heat equation}

\author{}
\date{}

\begin{document}
\maketitle

\centerline{\scshape Fredi Tr\"oltzsch}
\medskip
{\footnotesize
% please put the address of the first author
 \centerline{Institut f\"ur Mathematik}
   \centerline{Technische Universit\"at Berlin}
   \centerline{D-10623 Berlin, Germany}
} % Do not forget to end the {\footnotesize by the sign }

\medskip

\centerline{\scshape Daniel Wachsmuth \footnote{Daniel Wachsmuth was partially supported by the German Research Foundation DFG under project grant Wa 3626/1-1.}}
\medskip
{\footnotesize
 \centerline{Institut f\"ur Mathematik}
   \centerline{Universit\"at W\"urzburg}
   \centerline{D-97974 W\"urzburg, Germany}
}

\bigskip

%The abstract of your paper
\begin{abstract}
An optimal boundary control problem for the one-dimensional heat equation  is considered.
The objective functional includes a standard quadratic terminal observation, a Tikhonov
regularization term with regularization parameter $\nu$, and the $L^1$-norm of the control
that accounts for sparsity. The switching structure of the optimal control is discussed for
$\nu \ge 0$.  Under natural assumptions,  it is shown that the set of switching points of the optimal control
is countable with the final time as only possible accumulation point. The convergence of switching
points is investigated for $\nu \searrow 0$.
\end{abstract}

\renewcommand{\theenumi}{(\roman{enumi})}
\renewcommand{\labelenumi}{\theenumi}

\section{Introduction}
In this paper, we investigate the switching behavior of optimal controls for the following
sparse optimal control problem with terminal observation:

\begin{equation} \label{E1.1}
\min J(y,u) := \frac{1}{2} \int_\Omega |y(x,T) - y_\Omega(x)|^2 \dx + \frac{\nu}{2} \int_0^T |u(t)|^2\, \dt + \mu \int_0^T |u(t)|\, \dt
\end{equation}
subject to the parabolic initial-boundary value problem
\begin{equation}\label{E1.2}
\begin{array}{rcll}
y_t - \Delta y &=& 0 &\mbox{ in } (0,1) \times (0,T)\\
y_x(0,t) &=& 0 &\mbox{ in } (0,T)\\
y_x(1,t) + \alpha \, y(1,t) &=& u(t)&\mbox{ in } (0,T)\\
y(x,0) &=& 0&\mbox{ in } (0,1)
\end{array}
\end{equation}
and to the pointwise control constraints
\begin{equation} \label{E1.3}
a  \le u(t) \le b, \quad \mbox{ a.e.\@ in } [0,T].
\end{equation}

Bang-bang and switching properties for the solutions of optimal boundary control problems were extensively discussed in the 70ties.
%For $\mu = 0$, problems of this class were discussed  in the 70ties in the context of the well-known bang-bang principle.
If $\nu = \mu = 0$, then it is well known that the optimal control is of bang-bang type provided that $y_\Omega$ is not attained by the optimal state. This result was discussed in several papers for linear parabolic equations, see
\cite{glasac77,grusac80,sch80a}, cf. also \cite{tro10book}. For the case of the maximum norm as objective functional, the finite bang-bang principle was proved in \cite{glawec76}.
Bang-bang principles for nonlinear parabolic equations were discussed in \cite{sac78,sch89}.

For $\nu > 0$ but $\mu = 0$, the switching behavior of optimal controls was investigated in \cite{epptro86,tro87}.
%\cite{epptro86,tro84b,tro87}.
In particular,
the convergence of switching points for $\nu \to 0$ was addressed. Numerical examples and numerical methods exploiting the switching structure were  presented in
\cite{dhatro11,mac82,schittkowski1979}. Bang-bang properties for time-optimal parabolic boundary control
problems were studied, e.g., in \cite{fat76,MizelSeidman1997,WangWang2007,KunischWang2016,WangYan2016}.
This list of references on bang-bang principles
and switching properties is by no means exhaustive. We also refer to the references of these papers.

The main novelty of our paper is the discussion of the switching structure for sparse optimal controls of
parabolic boundary control problems (i.e., for the case $\mu > 0$). To our best knowledge, the switching properties of
sparse optimal boundary
controls for parabolic problems  were not yet discussed in the literature. In particular, this refers to the convergence of switching points
for the limit $\nu \to 0$.
In addition to proving convergence of switching points, we obtain also convergence rates with respect to $\nu$ for the approximation of switching points
for $\nu\searrow0$.

However, the general bang-bang structure of optimal sparse controls has already been investigated
in a sequence of papers on semilinear elliptic control problems. We refer to \cite{Casas2012}.
Our paper was inspired by these general results.

\begin{assumption}[Data]{\em
In this setting, real numbers $T > 0$, $\nu \ge 0$, $\mu > 0$, $\alpha \ge 0$, and $a < 0 < b$ are fixed. The sign restrictions on
$a$ and $b$ are needed only for some of the structural properties of the optimal control, neither for the existence of optimal controls nor for the necessary optimality conditions. The parameter $\mu$ is the so called sparse parameter.

Moreover, we fix a desired final state function $y_\Omega \in C[0,1]$. We require $y_\Omega \in C[0,1]$ in order to have continuity of the adjoint state up to the boundary. (For the optimality conditions, $y_\Omega \in L^2(0,1)$ would suffice.)}
\end{assumption}

%%%%%%%%%%%%%%%%%%%%%%%%%%%%%%%%%%%%%%
\section{Well-posedness of the problem and necessary optimality conditions}
%%%%%%%%%%%%%%%%%%%%%%%%%%%%%%%%%%%%%%

\setcounter{equation}{0}

\subsection{Well-posedness of the problem}

The linear initial-boundary value problem \eqref{E1.1} is well posed. For each function $u \in L^2(0,T)$, there exists a unique
solution $y \in W(0,T) \cap C(\bar Q)$, where $Q:= (0,1) \times (0,T)$ denotes the open space-time cylinder. Let us denote the
state $y$ associated with $u$ by $y_u$. The control-to-state mapping $u \mapsto y_u$ is linear and continuous from $L^2(0,T)$
to $W(0,T) \cap C(\bar Q)$.

Let us introduce the functionals $j : L^1(0,T) \to \mathbb{R}$,
\[
 j(u) := \int_0^T |u(t)|\, \dt = \|u\|_{L^1(0,T)}
\]
and $f\nu$,
\[
f_\nu(u) := \frac{1}{2} \int_\Omega |y_u(x,T) - y_\Omega(x)|^2 \dx + \frac{\nu}{2} \int_0^T |u(t)|^2\, \dt.
\]
Then the reduced objective functional $F$ is given by
\[
F_\nu(u) := J(y_u,u) = f_\nu(u) + \mu \, j(u).
\]
Introducing the set of admissible controls by
\[
 U_{ad} := \{ u \in L^2(0,T): a \le u(t) \le b, \ \mbox{ for a.a. } t \in [0,T]\},
\]
we can re-write the optimal control problem \eqref{E1.1}--\eqref{E1.3} in the short form
\[
 \tag{$P_\nu$} \min_{u \in  U_{ad}} F_\nu(u).
\]
The functional $F_\nu$ is continuous and convex, hence weakly lower semicontinuous. Moreover, the set $U_{ad}$
is weakly compact and non-empty. Therefore, there exists at least one optimal control of the problem ($P_\nu$),
which will be denoted by $u_\nu$ to indicate the correspondence to the Tikhonov parameter $\nu$.
By $y_\nu :=
y_{u_\nu}$ we denote the optimal state associated with $u_\nu$. In the case $\nu = 0$ we will drop the index $0$
and  write $\bar u := u_0$ and $\bar y := y_0$.

If $\nu > 0$, then $F_\nu$ is strictly convex and hence in this case the optimal control is unique.
Under a natural assumption, we show later that this uniqueness
also holds for $\mu > 0$.
In addition, in the case $\nu=0$ the optimal state is uniquely determined due to the strict convexity
of $f_\nu$ with respect to $y_u$.

%%%%%%%%%%%%%%%%%%%%%%%%%%%%%%%%%
\subsection{Fourier expansion for \eqref{E1.2}, Green's function}
%%%%%%%%%%%%%%%%%%%%%%%%%%%%%%%%%

For convenience of the reader, we recall some known facts on the representation of the weak solution by
a Green's function $G$. We consider the inhomogeneous initial-boundary value problem
\be \label{E3.12}
\begin{array}{rcll}
y_t(x,t)-y_{xx}(x,t)&=&f(x,t)&\mbox{ in }Q\\[0.5ex]
y_x(0,t)&=&0&\mbox{ in }(0,T)\\[0.5ex]
y_x(1,t)+\alpha \, y(1,t)&=&u(t)&\mbox{ in }(0,T)\\[0.5ex]
y(x,0)&=&y_0(x)&\mbox{ in }(0,1),
\end{array}
\ee
where $f\in L^2(Q)$, $y_0\in L^2(0,1)$, and
$u\in L^2(0,T)$ are given. We mention also for the case $\alpha = 0$, but later we will concentrate on
positive $\alpha$.
The weak solution of \eqref{E3.12} can be represented by a Green's function $G=G(x,\xi,t)$ as
\be \label{E3.13}
\begin{array}{rcl}
y(x,t)&=&\displaystyle \int_0^1 G(x,\xi,t)\, y_0(\xi)\d\xi+
\int_0^t \int_0^1 G(x,\xi,t-s)\, f(\xi,s)\, \d\xi \ds\\
&&+\displaystyle \int_0^t G(x,1,t-s)\, u(s)\, \ds,
\end{array}
\ee
where $G$ is given by the following {Fourier} expansions:
\be \label{fourierreihen}
G(x,\xi,t)=\left\{
\begin{array}{ll}
\displaystyle1+2\sum\limits_{n=1}^\infty\cos(n\pi x)
\cos(n\pi\xi)\exp(-n^2\pi^2t)&\mbox{ for } \alpha=0\\[0.5ex]
\displaystyle\sum\limits_{n=1}^\infty\frac{1}{N_n}\cos(\rho_nx)
\cos(\rho_n\xi)\exp(-\rho_n^2t) &\mbox{ for } \alpha>0.
\end{array} \right.
\ee
Here, $(\rho_n)$ is the monotone increasing sequence of non-negative solutions to the equation
\[
\rho \tan \rho=\alpha
\]
and
\[
N_n=\int_0^1\cos^2(\rho_n x)\dx = \frac12+\frac{\sin(2\rho_n)}{4\rho_n} = \frac12+\frac{\sin^2(\rho_n)}{2\alpha}
\]
are normalizing constants.  The numbers $n\, \pi$ and $\rho_n$ are the
eigenvalues of the differential operator $\partial^2/\partial x^2$ subject to the homogeneous boundary conditions
in \eqref{E3.12} for $\alpha=0$ and $\alpha>0$, respectively. For the eigenvalues, we know that $\rho_n \sim (n-1)\pi, \, n \to \infty$.
The functions $x \mapsto \cos(n\pi x)$ and $x \mapsto \cos(\rho_nx)$  are associated eigenfunctions, respectively.
After normalization by the factors $N_n$, they form a complete orthonormal system in $L^2(0,1)$; cf.
\cite{tycsam64}.

Notice that, for our case $f = 0$ and $y_0 = 0$, the term $y(x,T)$ in the objective functional has
the series representation
\begin{equation} \label{Eseries_yT}
y(x,T) = \sum_{n=1}^\infty \frac{\cos{\rho_n}}{N_n} \int_0^T e^{-\rho_n^2(T-s)}u(s)\,ds.
\end{equation}

%%%%%%%%%%%%%%%%%%%%%%%%%%
\section{Necessary optimality conditions}
%%%%%%%%%%%%%%%%%%%%%%%%%%%

\setcounter{equation}{0}

\subsection{The variational inequality}
%%%%%%%%%%%%%%%%%%%%%%%%%%%

It is well known that the derivative of the differentiable functional $f_\nu$ can be represented in the form
\[
f'_\nu(u) v = \int_0^T ( \varphi_u(1,t) + \nu \, u(t))\, v(t)\, \dt,
\]
where $\varphi_u$ is the adjoint state associated with $u$. It is the unique weak solution to the adjoint equation
\begin{equation}
\begin{array}{rcll}
- \varphi_t - \Delta \varphi &=& 0&\mbox{ in } Q\\
\varphi_x(0,t) &=& 0&\mbox{ in } (0,T)\\
\varphi_x(1,t) + \alpha \, \varphi(1,t) &=& 0&\mbox{ in } (0,T)\\
\varphi(x,T) &=& y_u(x,T) - y_\Omega(x)&\mbox{ in } (0,1).
\end{array}
\end{equation}
Notice that $y_\Omega$ is assumed to be continuous. Moreover, the function
$x \mapsto y_u(x,T)$ is also continuous, because $u \in L^\infty(0,T)$. Therefore, we have $\varphi_u \in C(\bar Q)$
and the continuity of the function $t \mapsto \varphi_u(1,t)$ on $[0,T]$.

\begin{theorem}[Necessary optimality condition] \label{T1} Let $\nu \ge 0$ be given, and let $u_\nu
\in U_{ad}$ be optimal for the problem {\rm ($P_\nu$)}. Then there exists a function
$\lambda_\nu \in \partial j(u_\nu)$, such that the variational inequality
\begin{equation} \label{varineq}
\int_0^T (\varphi_\nu(1,t) + \nu u_\nu(t) + \mu \, \lambda_\nu(t))(u(t) - u_\nu(t))\dt \ge 0 \quad \forall u \in U_{ad}.
\end{equation}
is satisfied with the adjoint state $\varphi_\nu := \varphi_{u_\nu}$.
\end{theorem}
This result is completely standard for the case $\mu = 0$, where $F_\nu$ is smooth, see
\cite{lio71}.  If $\mu > 0$, the associated
tools from subdifferential calculus can be found, for instance  in \cite[Thm.\@ 3.1]{casas_herzog_wachsmuth2012}.
 It is fairly obvious how these methods can be transferred to our problem. Therefore, we omit the proof.
An equivalent formulation is obtained by replacing the subdifferential in \eqref{varineq} by
directional derivatives, which is
\begin{equation} \label{varineq_dir}
\int_0^T (\varphi_\nu(1,t) + \nu u_\nu(t))(u(t) - u_\nu(t)) \dt +   \mu j'(u_\nu; \ u-u_\nu) \ge 0 \quad \forall u \in U_{ad}.
\end{equation}
By a standard argument, we find that this variational inequality of integral type implies the following pointwise
inequality:
\begin{equation} \label{varineq_dir_pointw}
(\varphi_\nu(1,t) + \nu u_\nu(t))(v - u_\nu(t))  +   \mu j'(u_\nu(t); v -u_\nu(t)) \ge 0 \quad \forall a \le v \le b
\end{equation}
 for a.a.\@  $t \in (0,T)$.
 Here, we denote by $\jmath'(u,h)$ the directional derivative of
 the real function $\jmath(u):=|u|$ at $u \in \mathbb{R}$ in the direction $h \in \mathbb{R}$.

%%%%%%%%%%%%%%%%%%%%%%%%%%%
\subsection{The case $\nu = 0$, bang-bang-bang properties}
%%%%%%%%%%%%%%%%%%%%%%%%%%%

In the case $\nu = 0$, a detailed discussion of the variational inequality \eqref{varineq} leads to
nice structural properties of the optimal control $\bar u$. We recall that for $\nu = 0$ the optimal control and the optimal state
are denoted by $\bar u$ and $\bar y$, respectively.

Associated with $\bar u$, we introduce the following measurable sets
 \[
 \begin{array}{rcl}
E_+ &=& \{ t \in (0,T): \bar u(t) > 0\},\\
E_0 &=& \{ t\in (0,T) : \bar u(t) = 0\},\\
E_- &=& \{ t \in (0,T): \bar u(t) < 0\}.
\end{array}
\]
The pointwise discussion of the variational inequality
\eqref{varineq} yields the following result.
\begin{lemma} \label{L1} For almost all $t \in [0,T]$, the following implications hold true:
\begin{equation} \label{2}
t \in
\left\{
\begin{array}{rcl}
E_+ &  \Rightarrow&  \bar \varphi(1,t) + \mu \le 0,\\
E_- &  \Rightarrow&  \bar \varphi(1,t) - \mu \ge 0,\\
E_0&\Rightarrow&  \bar \varphi(1,t) + \mu \, \bar \lambda(t)= 0.
\end{array}
\right.
\end{equation}
\end{lemma}
\begin{proof}
For a.a. $t \in E_+$, we have $\bar u(t) > 0$, hence $\bar \lambda(t) = 1$. Therefore \eqref{varineq} almost everywhere implies
\[
 \bar \varphi(1,t) + \mu \, \bar \lambda(t) = \bar \varphi(1,t) + \mu \le 0.
\]
On $E_-$, the discussion is analogous. In  $E_0$, we have a.e. $a < \bar u(t) < b$,
hence the reduced gradient must vanish here, $\bar \varphi(1,t) + \mu \, \bar \lambda(t) = 0$ a.e. on $E_0$.
\end{proof}
This discussion showed how the function $t \mapsto \bar \varphi(1,t) + \mu \, \bar \lambda(t)$ depends on the
sign of $\bar u$. Another investigation will reveal the switching structure of $\bar u$ related to the function
$t \mapsto \bar \varphi(1,t) + \mu \, \bar \lambda(t)$.
To this end let us define the open sets
\[
\begin{array}{rcl}
\Phi_+ &=&\{ t \in (0,T): \bar \varphi(1,t) - \mu >  0\},\\
\Phi_0 &=&\{ t \in (0,T): |\bar \varphi(1,t)|  < \mu\},\\
\Phi_- &=&\{ t\in (0,T) : \bar \varphi(1,t) + \mu  < 0\}.
\end{array}
\]
\begin{lemma} \label{L1b} For almost all $t \in [0,T]$, the following implications hold true:
\begin{equation}
t \in
\left\{
\begin{array}{rcl}
\Phi_+ &  \Rightarrow&  \bar u (t) = a,\\
\Phi_- &  \Rightarrow&  \bar u(t)=0,\\
\Phi_0&\Rightarrow&  \bar u(t) = b.
\end{array}
\right.
\end{equation}
\end{lemma}
\begin{proof}
 We first prove the claim for $\Phi_+$.
  The equation $\bar u(t)=a$ is obtained as follows: We have $|\bar \lambda(t) |\le 1$, hence
  \[
   0 < \bar \varphi(1,t) - \mu \le \bar \varphi(1,t) + \mu \, \bar \lambda(t)
  \]
  is satisfied a.e.\@ in $\Phi_+$. Now the variational inequality \eqref{varineq} implies
$\bar u(t) = a$ a.e.  in  $\Phi_+$. The continuity of the function $t \mapsto \bar \varphi(1,t)$ yields that $\Phi_+$
is an open set.
The proof for $\Phi_-$ is analogous.
The statement of Lemma \ref{L1} shows that $\Phi_0 \cap (E_+ \cup E_-)$ has measure zero. Hence
it holds $\bar u=0$ almost everywhere on $\Phi_0$.
\end{proof}

%%%%%%%%%%%%%%%%%%%%%%%
\subsection{Switching points of $\bar u$}
%%%%%%%%%%%%%%%%%%%%%%%

The switching behavior of the optimal control depends on the solutions of the two equations
\[
\bar \varphi(1,t) + \mu = 0 \ \ \mbox{ and } \ \ \bar \varphi(1,t) - \mu = 0, \quad t \in [0,T].
\]
To estimate their number, we need the following result:
\begin{lemma} \label{L2}
 The function $t \mapsto \bar \varphi(1,t)$ is continuous in $[0,T]$. It can be extended  to a
 holomorphic function in the complex half plane  $\{ z \in \mathbb{C} : \Re(z) < T\}$.
\end{lemma}
\begin{proof} The continuity of the function $t \mapsto \bar \varphi(1,t)$  follows
from  $\bar y \in C(\bar Q)$ and $y_\Omega \in C[0,1]$. Therefore, the terminal data for $\bar \varphi(T)$
are continuous in $[0,1]$, hence $\varphi \in  C(\bar Q)$ and the function $t \mapsto \bar \varphi(1,t)$
is continuous. We refer to \cite{casas97,rayzid98},  and \cite[Thm.\@ 5.5]{tro10book}.

The fact that the function $t \mapsto \bar \varphi(1,t)$ can be extended to a holomorphic function follows from its
the Fourier expansion. By the transformation of time $\tau := T-t$, we find from \eqref{E3.13} and \eqref{fourierreihen}
that
\be \label{series1}
\begin{array}{rcl}
\bar \varphi(1,t) &=& \displaystyle \int_0^1 G(1,x,T-t) (\bar y(x,T) - y_\Omega(x))\, \dx \\
&= &  \displaystyle
\sum_{n=1}^\infty \frac{\cos{(\rho_n)}}{N_n} \, e^{-\rho_n^2(T-t)}\, \int_0^1 d(x) \cos{(\rho_n x) }\, \dx,
\end{array}
\ee
where
\[
d(x) := \bar y(x,T) - y_\Omega(x).
\]
The eigenvalues $\rho_n$ behave asymptotically like $(n-1) \pi $, $n \to \infty$. Therefore,
the factor $e^{-\rho_n^2(T-t)}$  converges very fast to zero as $n \to \infty$ provided that
$t < T$. For $t \le T-\varepsilon$, $\varepsilon > 0$ fixed, the convergence of the series \eqref{series1} and of all
of its derivatives w.r. to $t$ is uniform in $-\infty < t \le T-\varepsilon$. The same holds true for the complex
extension
\[
\bar \varphi(1,z):= \sum_{n=1}^\infty \frac{\cos{(\rho_n)}}{N_n} \, e^{-\rho_n^2(T-z)}\, \int_0^1 d(x)
\cos{(\rho_n x)}\, \dx,
\]
if  $\Re(z) \le T-\varepsilon$. Therefore, the series defines a holomorphic function in the half plane
$\{ z \in \mathbb{C} : \Re(z) < T\}$ and $\bar \varphi(1,t)$ is obtained as its real part.
\end{proof}

Let us re-write the expansion of $\bar \varphi(1,t) $ in the shorter form
\begin{equation} \label{varphi1}
 \varphi(1,t) = \sum_{n=1}^\infty \frac{\cos{(\rho_n)}}{N_n}  \, e^{-\rho_n^2(T-t)} d_n,
\end{equation}
where the numbers
\[
 d_n := \int_0^1 d(\xi) \cos{(\rho_n \xi) }\d\xi
\]
correspond to the Fourier coefficients of $d$ (the exact Fourier coefficients are given by $d_n / \sqrt{N_n}$\, ).
The decisive result for the switching behavior of $\bar u$ is the following:
\begin{lemma} \label{L3} Let $0 < \varepsilon < T$ be given and assume $\|\bar y(\cdot,T) - y_\Omega(\cdot)\|_{L^2(0,1)} > 0$.
 Then the equations
\be \label{switchings}
\bar \varphi(1,t) + \mu = 0 \quad \mbox{ and } \quad \bar \varphi(1,t) - \mu = 0
\ee
have at most finitely many solutions in $[0,T-\varepsilon]$.

Therefore, in $[0,T]$ these equations have at most countably many solutions
that may accumulate only at $t=T$.
\end{lemma}
\begin{proof}Let us consider only the first equation,
\[
\bar \varphi(1,t) + \mu = 0,  \qquad t \in [0,T-\varepsilon].
\]
Assume to the contrary that it has infinitely many solutions in $[0,T-\varepsilon]$. Then they must have an
accumulation point $\bar t \in [0,T-\varepsilon]$. By the identity theorem for holomorphic functions, we deduce
\[
\bar \varphi(1,t) + \mu = 0 \quad \forall t \in (-\infty,T-\varepsilon].
\]
Differentiating this equation, we obtain $\frac\d{\dt} \bar\varphi(1,t) = 0$, hence
\be \label{series2}
\sum_{n=1}^\infty \frac{\cos{(\rho_n)}}{N_n}  \, \rho_n^2 e^{-\rho_n^2(T-t)} d_n = 0 \quad \mbox{ in } (-\infty,T-\varepsilon].
\ee
We multiply \eqref{series2} by $e^{\rho_1^2(T-t)}$ and get
\[
\frac{\cos{(\rho_1)}}{N_1}  \, \rho_1^2 \, d_1 + \sum_{n=2}^\infty
\frac{\cos{(\rho_n)}}{N_n}  \, \rho_n^2 e^{-(\rho_n^2- \rho_1^2)(T-t)} d_n = 0 \quad \mbox{ in }  (-\infty, T-\varepsilon].
\]
Now, we pass to the limit $t \to -\infty$. Since $\rho_n^2 > \rho_1^2$ holds for $n \ge 2$  and the series is
uniformly convergent,  it follows $\exp{(-(\rho_n^2- \rho_1^2)(T-t))} \to 0$, hence
\begin{equation} \label{cos_not_zero}
\frac{\cos{(\rho_1)}}{N_1}  \, \rho_1^2 \, d_1 = 0,
\end{equation}
and hence $d_1$ = 0.  Notice that $\cos(\rho_n) \not = 0$ holds for all $n \in \mathbb{N}$. Therefore, the first item in \eqref{series2} is zero. Multiplying \eqref{series2} by $e^{\rho_2^2(T-t)}$
and passing to the limit $t \to \infty$, we find $d_2= 0$. Repeating this method infinitely many times,
it follows $d_n = 0$ for all $n \in \mathbb{N}$.

The system of functions $\{ \cos{(\rho_n \cdot) } : n \in \mathbb{N}\}$ is complete in $L^2(0,1)$,
hence $d = 0$ must hold in the sense of $L^2(0,1)$.
This contradicts the assumption that $d = \bar y(\cdot,T)- y_\Omega \not = 0$.
\end{proof}

% In the preceding subsection, we established conditions on the reduced gradient related to the sign of $\bar u$.
% Now, we take another view. From the sign of the reduced gradient, we deduce the form of $\bar u$.
%
% We define
%  \[
%  \begin{array}{rcl}
% \Phi_+ &=&\{ t \in (0,T): \bar \varphi(1,t) - \mu >  0\}\\
% \Phi_- &=&\{ t\in (0,T) : \bar \varphi(1,t) + \mu  < 0\}.
% \end{array}
% \]
%
\begin{theorem} \label{T2}
 Suppose $\|\bar y(\cdot,T) - y_\Omega(\cdot)\|_{L^2(0,1)} > 0$.
 For each $\varepsilon > 0$, the sets $\Phi_+ \cap (0,T-\varepsilon)$, $\Phi_0 \cap (0,T-\varepsilon)$, and
$\Phi_- \cap (0,T-\varepsilon)$ are the union of finitely many open intervals. Consequently,  $\Phi_+$, $\Phi_0$, and $\Phi_-$
are the union of at most countably many open intervals (the components of $\Phi_+$, $\Phi_0$,  and $\Phi_-$) which can
accumulate only at $t=T$.
\end{theorem}
\begin{proof}
The continuity of the function $t \mapsto \bar \varphi(1,t)$ yields that $\Phi_+$
is an open set.
All components of  $\Phi_+$ are bounded by two zeros
of the function $t \mapsto \bar \varphi(1,t) - \mu$ in $[0,T-\varepsilon]$ or by the numbers $0$, $T-\varepsilon$.
By Lemma \ref{L3}, the number of solutions to the equation  $\bar \varphi(1,t) - \mu = 0$  in $(0,T-\varepsilon)$ is finite,
hence the number of components of $\Phi_+$  in $(0,T-\varepsilon)$ is finite, too. The statement on the accumulation
of components in $[0,T]$ is an obvious consequence.
The claim for $\Phi_0$ and $\Phi_-$ can be proven analogously.
 \end{proof}

\begin{figure}
\begin{center}
\end{center}
\caption{Switching structure of the optimal control $\bar u$}
\end{figure}

In addition, it follows that the complement of $\Phi_+\cup \Phi_0 \cup \Phi_+$, which is the set
of solutions of \eqref{switchings}, is countable. Hence, the switching conditions of Lemma \ref{L1b}
uniquely define $\bar u$ almost everywhere on $(0,T)$.
This implies that $\bar u$ almost everywhere attains values from the discrete set $\{a,0,b\}$.
Moreover, $\bar u$ is piecewise constant on $[0,T-\epsilon)$ for all $\epsilon>0$
with discontinuities only located at the solutions of \eqref{switchings}.
These points will  be called switching points in the sequel.

\begin{definition}
 All points  $t \in (0,T)$, where one of the two functions  $t \mapsto \bar \varphi(1,t) - \mu$ and
 $t \mapsto \bar \varphi(1,t) + \mu$ changes the sign, are said to be switching points of  $\bar u$.
\end{definition}

\begin{theorem}[Bang-Bang-Bang Principle] \label{T3} Assume $\|\bar y(\cdot,T) - y_\Omega\|_{L^2(0,1)} > 0$.
Then the following switching properties hold true:
\begin{enumerate}
 \item\label{T3_1} For each $0 < \varepsilon < T$, the number of switching points of $\bar u$  in
$[0,T-\varepsilon]$
is finite. Therefore, the number of switching points of $\bar u$ in $[0,T]$ is at most countable and switching points can only
accumulate at $t=T$.

Between two subsequent switching points, the optimal control  $\bar u$ is identically constant and equal to  one of the
values $b$, $a$ or $0$.

\item\label{T3_2} If $\bar \varphi(1,T) \not= \mu$ and  $\bar \varphi(1,T) \not= - \mu$, then the number of switching points of
$\bar u$ is finite and there is a sufficiently small $\delta > 0$ such that, for a.a. $t\in  (T-\delta,T]$,
\be
\bar u(t) =
\left\{
\begin{array}{rcl}
b &\mbox{ if }& \bar \varphi(1,T) <- \mu\\
0 &\mbox{ if }& -\mu < \bar \varphi(1,T) <  \mu\\
a &\mbox{ if } &\bar \varphi(1,T) > \mu
\end{array}
\right.
\ee
is fulfilled.

\item\label{T3_3} If $\bar \varphi(1,T) = \mu$, then, in a certain neighborhood $(T-\delta,T]$, the optimal control can switch
at most countably many times between  $a$ and $0$.  In the case $\bar \varphi(1,T) = - \mu$ it can
switch at most countably many times between  $b$ and $0$ in $(T-\delta,T]$.

\item\label{T3_4} Switching-over of $\bar u$ between  $a$ and $b$ cannot happen.
\end{enumerate}
\end{theorem}
 \begin{proof} \ref{T3_1} Switching points can only be boundary points of the sets $\Phi_+$, $\Phi_-$, and $\Phi_0$.
 Therefore, they must solve  one of the two equations \eqref{switchings}. By Lemma \ref{L3}, the number
 of their solutions is at most countable and can accumulate only at $t=T$. Between switching points, $\bar u$
 can only attain the values $b$, $a$, and $0$, cf. Lemma \ref{L1b}.  This proves (i).

 \ref{T3_2} If $\bar \varphi(1,T) \not= \mu$ and $\bar \varphi(1,T) \not= -\mu$, then we are in one of the cases $\bar \varphi(1,t)  > \mu$, $-\mu < \bar \varphi(1,t)  < \mu$, or $\bar \varphi(1,t)  < - \mu$ at $t=T$. In either case, by continuity of the function $t \mapsto \bar \varphi(1,t)$, these inequalities remain valid for all $t$ in a sufficiently small interval $(T-\delta,T]$.

 Now we apply Theorem \ref{T1}: In the first case, we have $ t \in \Phi_+$ and hence $\bar u = a$ in  $(T-\delta,T]$. In the second, we have $t \in \Phi_0$, hence $\bar u(t) = 0$ in  $(T-\delta,T]$, and in the third we obtain analogously that $\bar u(t) = b$ in  $(T-\delta,T]$.

 \ref{T3_3}, \ref{T3_4} If  $\bar \varphi(1,T) = \mu$ or  $\bar \varphi(1,T) = - \mu$ is satisfied, we cannot exclude an accumulation of switching points
 at $t=T$. By the continuity of $t \mapsto \bar \varphi(1,t)$, these can only be a switchover between
 $b$ and $0$ or $a$ and $0$, respectively.
 \end{proof}
\begin{theorem}\label{T4}If $\|\bar y(\cdot,T) - y_\Omega\|_{L^2(0,1)} > 0$, then the optimal control $\bar u$ is unique.
\end{theorem}
\begin{proof} Let two optimal controls $\bar u$ and $\bar v$ be given.
Due to strict convexity of $f_\nu$, the optimal state is unique, which gives $\bar y = y_{\bar u}=y_{\bar v}$.
Thanks to Theorem \ref{T3}, both
controls must be of
bang-bang-bang type: Almost everywhere and in open intervals, they admit only the values $a$, $b$ or
$0$. Since the control problem is convex, every convex combination $\theta \bar u + (1-\theta)\bar v$
is an optimal control and bang-bang-bang. This is only possible if $\bar u=\bar v$ holds almost everywhere.
\end{proof}

\begin{definition} Assume that $\|\bar y(\cdot,T) - y_\Omega\|_{L^2(0,1)} > 0$ holds for the solution of
(P$_0$). Then the optimal
control $\bar u$ is  has at most countably many switching points.
%that are the solutions of the two equations $\bar \varphi(1,t) = \pm \mu$.
The switching points in $(0,T)$ solving  $\bar \varphi(1,t) = \mu$
are denoted by $t_j^\mu$ and the ones solving $\bar \varphi(1,t) = - \mu$ are denoted by $t_j^{-\mu}$, $j \ge 1.$
These switching points are ordered such that  $t_j^\mu < t_{j+1}^\mu$ and $t_j^{-\mu}< t_{j+1}^{-\mu}$
holds for all $j \ge 1$.
\end{definition}

%%%%%%%%%%%%%%%%%%%%%%
\subsection{The case $\nu > 0$}
%%%%%%%%%%%%%%%%%%%%%%

Now we assume $\nu > 0$ and consider the problem (P$_\nu$), i.e., the problem
\be \label{Pnu}
\min_{u \in  U_{ad}} \left\{ \int_\Omega |y_u(x,T) - y_\Omega(x)|^2 \dx + \mu \int_0^T |u(t)|\, \dt+ \frac{\nu}{2} \int_0^T|u(t)|^2\, \dt\right\}.
\ee
We recall that $y_u$ is defined as solution of the equation \eqref{E1.2} associated to $u$.

Again, this problem has an optimal control $u_\nu$ with associated optimal state $y_\nu:= y_{u_\nu}$.
By strict convexity of the
functional in \eqref{Pnu}, the optimal control is unique. The associated adjoint state is $\varphi_\nu := \varphi_{u_\nu}$.
The necessary optimality condition is stated  in Theorem \ref{T1}.

By a detailed pointwise discussion of the variational inequality \eqref{varineq}, the following result is deduced completely analogous to a result of
\cite{casas_herzog_wachsmuth2012,Stadler2009} for
a class of elliptic equations.
We also refer to a later result for a parabolic problem in \cite{casas_ryll_troeltzsch2014}.
In the theorem, the projection function $\mathbb{P}_{[s_1,s_2]}: \mathbb{R} \to [s_1,s_2]$ is defined by
\[
\mathbb{P}_{[s_1,s_2]}(s) = \max\{s_1,\min\{s,s_2\}\}.
\]
\begin{theorem} \label{T5} For almost all $t \in [0,T]$, the following equations are fulfilled:
\begin{eqnarray}
u_\nu(t) &=& \mathbb{P}_{[a,b]}\left( -\frac{1}{\nu}(\varphi_\nu(1,t) + \mu \, \lambda_\nu(t))\right), \label{projection1}\\[1ex]
u_\nu(t) &=& 0 \ \ \mbox{ if and only if }  \quad |\varphi_\nu(1,t)| \le \mu, \label{sparsity}\\[1ex]
\lambda_\nu(t) &= &\mathbb{P}_{[-1,1]}\left(- \frac{1}{\mu}\, \varphi_\nu(1,t)\right). \label{projection2}
\end{eqnarray}
\end{theorem}
The relation \eqref{sparsity} expresses the sparsity of the optimal control, while \eqref{projection2} extracts a single element
out of the subdifferential of $j(u_\nu)$. We skip the proof, because it is completely analogous to the one in
\cite{casas_herzog_wachsmuth2012}.

As a simple conclusion Theorem \ref{T5} we get that, for $\nu > 0$, the functions $\lambda_\nu$
and $u_\nu$ are continuous on $[0,T]$: Indeed, the function $t \mapsto \varphi_\nu(1,t)$ is continuous,
hence \eqref{projection2} yields the continuity of $\lambda_\nu$. Inserting this in \eqref{projection1}, we
see the continuity of $u_\nu$.

Let us determine the structure of $u_\nu$. We might follow the presentation in
\cite{casas_ryll_troeltzsch2015}, but for the convenience of the reader we prove the results again in our framework.
Inserting \eqref{projection2} in \eqref{projection1},
we find
\begin{equation} \label{projection_summary}
u_\nu(t) = \mathbb{P}_{[a,b]}\left( -\frac{1}{\nu}\left(\varphi_\nu(1,t) + \mu \,
\mathbb{P}_{[-1,1]}\left(- \frac{1}{\mu}\, \varphi_\nu(1,t)\right)\right)\right).
\end{equation}
Discussing this representation, we find the following result:
\begin{theorem} \label{T6} Assume $\nu > 0$. Then the implications
\begin{eqnarray}
 \varphi_\nu(1,t) \in (-\infty,-\mu - \nu b) &\Rightarrow& u_\nu(t) = b \label{impl1}\\
 \varphi_\nu(1,t) \in (-\mu - \nu b,-\mu) &\Rightarrow& u_\nu(t) = -\frac{1}{\nu}(\varphi_\nu(1,t) + \mu)\label{impl2}\\
 \varphi_\nu(1,t) \in (-\mu, \mu) &\Rightarrow& u_\nu(t) = 0\label{impl3}\\
 \varphi_\nu(1,t) \in (\mu,\mu - \nu a) &\Rightarrow& u_\nu(t) = -\frac{1}{\nu}(\varphi_\nu(1,t) - \mu)\label{impl4}\\
 \varphi_\nu(1,t) \in (\mu - \nu a,\infty) &\Rightarrow& u_\nu(t) = a\label{impl5}
\end{eqnarray}
hold almost everywhere in $[0,T]$.
 \end{theorem}
 \begin{proof}
  (a) The implication  \eqref{impl3}  follows immediately from \eqref{sparsity}.

  (b) Now we show \eqref{impl2}. Here, the inclusion for $\varphi_\nu$ is equivalent to
  \begin{equation} \label{incl+}
 1 < -\frac{1}{\mu}  \varphi_\nu(1,t) < 1 + \frac{\nu b}{\mu}.
  \end{equation}
  The left-hand side implies that $\mathbb{P}_{[-1,1]}(-\frac{1}{\mu} \varphi_\nu(1,t) ) = 1$, hence \eqref{projection_summary}
  yields
  \begin{equation}\label{projection3}
   u_\nu(t) = \mathbb{P}_{[a,b]} (-\frac{1}{\nu}  (\varphi_\nu(1,t) + \mu)).
  \end{equation}
  The last  inequality of \eqref{incl+} is equivalent with  $-\frac{1}{\nu} (\varphi_\nu(1,t) + \mu)) < b$, hence
  \[
   u_\nu(t) = -\frac{1}{\nu}  (\varphi_\nu(1,t) + \mu),
  \]
  i.e., \eqref{impl2} is shown.
  (c) To prove \eqref{impl1}, we mention that the inclusion $\varphi_\nu(1,t) \in (-\infty,-\mu - \nu b) $
  is equivalent with
  \[
  -\frac{1}{\mu} \varphi_\nu(1,t) ) > 1 + \frac{\nu}{\mu}.
  \]
  Again, we  arrive at \eqref{projection3}. However, the inclusion above also yields that
  $
   -\frac{1}{\nu}  (\varphi_\nu(1,t) + \mu) > b,
  $
  and hence from \eqref{projection3} we obtain the conclusion of \eqref{impl1}.

  The implications \eqref{impl4} and \eqref{impl5} can be confirmed in the same way.
 \end{proof}

 This theorem reveals that the solutions of the four equations
 \begin{eqnarray}
  \varphi_\nu(1,t) + \mu + \nu b &= &0\label{sw1}\\
  \varphi_\nu(1,t) + \mu &= &0\label{sw2}\\
  \varphi_\nu(1,t) - \mu&=& 0\label{sw3}\\
  \varphi_\nu(1,t) -\mu +  \nu a &=&   0   \label{sw4}
  \end{eqnarray}
 determine the switching behavior of $u_\nu$. In other words, $u_\nu$ can only switch
 in the zeros of the four functions  standing in the left-hand side of \eqref{sw1}--\eqref{sw4}.
\begin{definition}
 Any $t \in (0,T)$, where one of the functions  in the left-hand side of \eqref{sw1}--\eqref{sw4} changes its sign, is said to be a switching point of $u_\nu$.
\end{definition}
\begin{lemma} If $\|y_\nu(\cdot,T)-y_\Omega\|_{L^2(0,1)} > 0$, then each of the equations \eqref{sw1}--\eqref{sw4}
can have at most countably
many solutions that can accumulate only at $t=T$. Therefore,  $u_\nu$ can have at most countably
many switching points that can accumulate only at $t=T$.
 \end{lemma}
 \begin{proof}
  The proof is almost identical with that of Lemma \ref{L2}, since the adjoint state $\varphi_\nu$ solves the same
  adjoint equation as $\bar \varphi$, but with terminal value $y_\nu(\cdot,T)-y_\Omega$. Therefore,
  we have
  \[
\begin{array}{rcl}
\varphi_\nu(1,t) &=& \displaystyle \int_0^1 G(1,\xi,T-t) (\bar y(\xi,T) - y_\Omega(\xi))\,\d\xi \\
&= &  \displaystyle
\sum_{n=1}^\infty \frac{\cos{(\rho_n)}}{N_n} \, e^{-\rho_n^2(T-t)}\, \int_0^1 d_\nu(\xi) \cos{(\rho_n \xi) }\,\d\xi,
\end{array}
\]
where $d_\nu = y_\nu(\cdot,T)-y_\Omega$. Now we proceed as in the proof of Lemma \ref{L2}.
  \end{proof}
  \begin{definition} The switching points of $u_\nu$, i.e., the solutions in $(0,T)$ of \eqref{sw1},
  \eqref{sw2}, \eqref{sw3}, and \eqref{sw4}, where the associated function changes its sign, are denoted
  by $t_{j}^{-\mu-\nu b}, \, t_{j}^{-\mu,\nu}, \, t_{j}^{\mu,\nu}$, and $t_{j}^{\mu-\nu a}$, respectively,
  for all $j \ge 1$ that may occur.

 Since the set of  switching points is countable with possible accumulation only at $t=T$, we can assume that the sequences of switching points are ordered w.r. to $j$, namely  $t_{j}^* < t_{j+1}^*$ holds for all $j \ge 1$ that appear and  $* \in \{-\mu-\nu b,(-\mu,\nu),(\mu,\nu),\mu-\nu a\}$.
\end{definition}

%%%%%%%%%%%%%%%%%%%%%%%
\section{Pass to the limit $\nu \to 0$}
%%%%%%%%%%%%%%%%%%%%%%%%

\setcounter{equation}{0}

In this section, we discuss the convergence of controls $u_\nu$ for $\nu\searrow 0$.
In addition, the convergence of switching points can be shown.
First, we discuss the convergence of the sequences
$(u_\nu)$, $(y_\nu)$
and $(\varphi_\nu)$ of optimal quantities for (P$_\nu$).
\begin{theorem}\label{T7} (i)  For $\nu \searrow 0$, the sequence  $(u_\nu)$ contains a strongly convergent subsequence
in $L^2(0,1)$, denoted w.l.o.g.\@ by $(u_\nu)$ again, such that $u_\nu \rightharpoonup \hat u$, $\nu \searrow 0$.
The control $\hat u$ is optimal for  (P$_0$). The associated subsequences $(y_{u_\nu})$ and $(\varphi_{u_\nu})$
converge uniformly in $\bar Q$ to $y_{\hat u}$ and $\varphi_{\hat u}$, respectively, as $\nu \searrow 0$.

(ii) If the optimal control $\bar u$ of  (P$_0$) is unique, then these convergence properties hold for the whole
sequences $(u_\nu)$,  $(y_{u_\nu})$, and $(\varphi_{u_\nu})$.
 \end{theorem}
 \begin{proof}
  (i) Since $U_{ad}$ is weakly compact, the existence of a weakly convergent subsequence $(u_\nu)$ with weak
  limit $\hat u \in U_{ad}$ is obvious.
  Let $\bar u$ be optimal for  (P$_0$). Then we have
  \begin{equation} \label{E4.1}
   F_\nu(\bar  u) \ge F_\nu(u_\nu) \ge F_0(u_\nu) \qquad \forall \nu > 0.
  \end{equation}
  Passing to the limit, we obtain from \eqref{E4.1}
  \[
   F_0(\bar u) = \lim_{\nu \searrow 0} F_\nu(\bar u)
   \ge \lim \inf _{\nu \searrow 0}F_0(u_\nu) \ge F_0(\hat u),
  \]
where we used the weak lower semicontinuity of $F_0$ to get  the last inequality. Therefore, $\hat u$ must also be optimal
for  (P$_0$).

From optimality of $u_\nu$ and $\hat u$ for (P$_\nu$) and (P$_0$), respectively, we get
\begin{align*}
 F_\nu(u_\nu) \le F_\nu(\hat  u) &  =
 \frac{1}{2} \int_\Omega |y_{\hat u }(x,T) - y_\Omega(x)|^2 \dx + \mu j(\hat u) +
 \frac{\nu}{2} \int_0^T |\hat u(t)|^2\, \dt  \\
 &\le  \frac{1}{2} \int_\Omega |y_{u_\nu}(x,T) - y_\Omega(x)|^2 \dx + \mu j(u_\nu) +
 \frac{\nu}{2} \int_0^T |\hat u(t)|^2\, \dt .
\end{align*}
This implies
\[
  \frac{\nu}{2} \int_0^T |u_\nu(t)|^2\, \dt \le  \frac{\nu}{2} \int_0^T |\hat u(t)|^2\, \dt
\]
and hence, dividing by $\nu/2$ we obtain
\[
 \int_0^T |\hat u(t)|^2\, \dt \le \liminf_{\nu \searrow0} \int_0^T
 |u_\nu(t)|^2\, \dt \le \limsup_{\nu \searrow0} \int_0^T|u_\nu(t)|^2\, \dt
 \le \int_0^T |\hat u(t)|^2\, \dt.
\]
This implies convergence of norms and strong convergence $u_\nu\to \hat u$ in $L^2(0,T)$ for $\nu\searrow 0$.
The strongly convergent subsequence $(u_\nu)$ in $L^2(0,1)$ is transformed to a uniformly convergent subsequence
 $(y_{u_\nu})$, i.e.,  $y_{u_\nu} \to y_{\hat u}$ in $C(\bar Q)$.
  Therefore, we also have $\varphi_{u_\nu} \to \varphi_{\hat u}$ in $C(\bar Q)$, because the mapping associating the solution
 of the adjoint equation to the final datum is continuous  from $C[0,1]$ to  $C(\bar Q)$ and we have
 $\varphi_{u_\nu}(T) =y_{u_\nu}(T) - y_\Omega$.

 (ii) If the optimal control  of  (P$_0$) is unique, say $\bar u$, then all subsequences of $u_\nu$ contain a subsequence
 converging weakly to the same limit $\bar u$. Then the whole sequence $(u_\nu)$ converges to $\bar u$.
 This transfers to the sequences $(y_{u_\nu})$ and $(\varphi_{u_\nu})$.
\end{proof}

\begin{lemma}\label{lem6}
For each $k\in \mathbb N$ and $\varepsilon \in (0,T)$, there is a constant $c>0$ such that
\[
\left\| \frac{\d^k}{\dt^k}\varphi_{u_\nu}(1,\cdot)-\frac{\d^k}{\dt^k}\varphi_{\hat u}(1,\cdot)\right\|_{C([0,T-\varepsilon])}
\le c  \ \|y_{u_\nu}(\cdot,T) -y_{\hat u}(\cdot,T) \|_{L^2(\Omega)}
\]
holds for  all $\hat u\in U_{ad}$.
\end{lemma}
\begin{proof}
  In $[0,T-\varepsilon]$, the formally differentiated Fourier series is given by
 \[
\frac{\d}{\dt} \varphi_{u_\nu}(1,t) = \sum_{n=1}^\infty \rho_n^2\, \frac{\cos{\rho_n}}{N_n}\int_0^1\cos{(\rho_n x)}
(y_{u_\nu}(x,T)-y_\Omega(x))\dx  \, e^{-\rho_n^2(T-t)}.
 \]
 It  is uniformly convergent, since the series
 \[
 ( M + \|y_\Omega\|_{C[0,1]})\sum_{n=1}^\infty \frac{1}{N_n}\rho_n^2 \, e^{-\rho_n^2\varepsilon}
 \]
 with $M = \sup_{u \in U_{ad}}\|y_u\|_{C(\bar Q)}$ is a convergent majorant.
 Let $\hat u\in U_{ad}$ be given. Then $\frac{\d}{\dt} \varphi_{\hat u}(1,t)$ has an analogous series representation.
 Hence, we can estimate
 \begin{multline*}
  \left|\frac{\d}{\dt} \varphi_{u_\nu}(1,t) - \frac{\d}{\dt} \varphi_{\hat u}(1,t)\right| \\
  = \left|\sum_{n=1}^\infty \rho_n^2\, \frac{\cos{\rho_n}}{N_n}\int_0^1\cos{(\rho_n x)}
(y_{u_\nu}(x,T)-y_{\hat u}(x))\dx  \, e^{-\rho_n^2(T-t)}\right|\\
  \le \|y_{u_\nu}(\cdot,T) -y_{\hat u}(\cdot,T) \|_{L^2(\Omega)} \sum_{n=1}^\infty \rho_n^2\, \frac{1}{\sqrt{N_n}} \, e^{-\rho_n^2(T-t)},
 \end{multline*}
 which proves the claim for $k=1$.
 The proof can be completed by an induction argument with respect to $k$.
\end{proof}

The norm $\|y_{u_\nu}(\cdot,T) -y_{\hat u}(\cdot,T) \|_{L^2(\Omega)}$ can be estimated
with the help of the following result.

\begin{lemma}
 Let $\hat u$ be optimal for (P$_0$). Then it holds
 \[
\|y_{u_\nu}(\cdot,T) -y_{\hat u}(\cdot,T) \|_{L^2(\Omega)}^2 + \nu \, \|u_\nu -\hat u\|_{L^2(0,T)}^2
\le \nu (\hat u,\hat u -u_\nu)_{L^2(0,T)}.
\]
\end{lemma}
\begin{proof}
 This is an immediate consequence of the optimality conditions, see, e.g., \cite[Lemma 2.5]{WachsmuthWachsmuth2011}.
\end{proof}

Combining these two results, we obtain a convergence rate for $\nu\searrow0$ for the adjoint states.

\begin{lemma}\label{lem8}
 Let $\hat u$ be optimal for (P$_0$). Then
for each $k\in \mathbb N$ and $\varepsilon \in (0,T)$, there is a constant $c>0$ such that
\[
\left\| \frac{\d^k}{\dt^k}\varphi_{u_\nu}(1,\cdot)-\frac{\d^k}{\dt^k}\varphi_{\hat u}(1,\cdot)\right\|_{C([0,T-\varepsilon])}
\le c \ \nu^{1/2}
\]
holds for all $\nu>0$.
\end{lemma}
\begin{proof}
 This is consequence of the previous two lemmas and the boundedness of $U_{ad}$.
\end{proof}

Now we are able to prove the convergence of switching points of
optimal controls for $\nu \searrow 0$.

\begin{theorem}
\label{T9}
Assume that $\|\bar y(\cdot,T) - y_\Omega\|_{L^2(0,1)} > 0$ is fulfilled.
 Let $t_j^\mu\in (0,T)$ be such that $\bar \varphi(1,t_j^\mu) = \mu$,
 and let $n$ be the smallest positive integer such that
\begin{equation} \label{E:slope_condition}
 \frac{\d^n}{\dt^n} \bar\varphi(1,t_j^\mu) \ne 0.
\end{equation}

Then there are $\nu_0>0$ and $\tau>0$ with the following properties:
\begin{enumerate}
 \item
For all $\nu\in (0,\nu_0)$, each of the equations
 \eqref{sw3} and \eqref{sw4} has at most $n$ solutions in the interval $(t_j^\mu-\tau,t_j^\mu+\tau)$.
 The equations \eqref{sw1} and \eqref{sw2} do not have solutions in $(t_j^\mu-\tau,t_j^\mu+\tau)$.
\item
Additionally, there is $c>0$ such that
\begin{align}
 |t_j^\mu - t_\nu| &\le c \ \nu^{\frac1{2n}},  \label{eq43}\\
 |t_j^\mu - t_\nu| &\le c \  (\|y_{u_\nu}(\cdot,T) -y_{\hat u}(\cdot,T) \|_{L^2(\Omega)}^{\frac1n} + \nu^{\frac1n}),
 \label{eq44}
\end{align}
for all $\nu<\nu_0$ and all $t_\nu\in (t_j^\mu-\tau,t_j^\mu+\tau)$ solving one of the equations \eqref{sw3}--\eqref{sw4}.
\item
If $n$ is odd, in particular if $t_j^\mu$ is a switching point of $\bar u$, then there is $\nu_1\in (0,\nu_0)$ such that for all $\nu\in (0,\nu_1)$ there exist switching points of
$u_\nu$  in the interval $(t_j^\mu-\tau,t_j^\mu+\tau)$ that solve one of \eqref{sw3}--\eqref{sw4}.
\end{enumerate}
Analogous results hold for solutions $t_i^{-\mu}$ of $\bar \varphi(1,t) = -\mu$.
\end{theorem}

\begin{proof}
Denote $\bar t:=t_j^\mu$.  Assume for the moment that
$\Lambda:=\frac{\d^n}{\dt^n} \bar\varphi(1,\bar t) > 0$ holds in \eqref{E:slope_condition}.
Take $\tau>0$ such that $\bar t+ \tau < T$
and
\[
  \bar\varphi(1,t) \ge \frac\mu2
\]
as well as
\be\label{eq043}
 \frac{\d^n}{\dt^n} \bar\varphi(1,t) \ge  \frac\Lambda2> 0
\ee
are satisfied for all $t\in (0,T)$ with $|t-\bar t|<\tau$.
Thanks to Lemma \ref{lem6}, we have uniform convergence of $\frac{\d^n}{\dt^n} \phi_\nu(1,\cdot)$ to $\frac{\d^n}{\dt^n}
\bar\varphi(1,\cdot)$
in $[\bar t-\tau,\bar t+\tau]$.  Therefore,
there is $\nu_0>0$ such that
\[
\phi_\nu(1,t) \ge 0
\]
and
\[
 \frac{\d^n}{\dt^n} \phi_\nu(1,t)  \ge \frac\Lambda4> 0
\]
hold for all $\nu\in(0,\nu_0)$ and $t\in (0,T)$ with $|t-\bar t|<\tau$.
Hence, the equation $\phi_\nu(1,t)=\mu$ can have at most $n$ distinct solutions in $(t-\tau,t+\tau)$.
Analogously, the equation $\phi_\nu(1,t)=\mu-\nu a$ has  at most $n$ distinct solutions in $(t-\tau,t+\tau)$.
In addition,  by $\varphi_\nu(1,t) \ge 0$ in $(t-\tau,t+\tau)$,  the equations
$\phi_\nu(1,t)=-\mu$ and $\phi_\nu(1,t)=-\mu-\nu b$  do not have  solutions  in $(t-\tau,t+\tau)$.

Let now $t\in (0,T)$ be given with $|t-\bar t|<\tau$.
Then, performing a Taylor expansion and invoking the  assumption, we find
\be\label{eq45}
 \bar \phi(1,t) = \mu + \frac1{n!} \frac{\d^n}{\dt^n} \bar\varphi(1,\xi) (t-\bar t)^n
\ee
with some intermediate point $\xi$.

Let $n$ be an odd integer.
Setting $t=\bar t-\tau$ and $t=\bar t+\tau$ and taking \eqref{eq043} into account,
yields
\begin{align*}
  \bar \phi(1,\bar t-\tau) -\mu &\le  - \frac\Lambda{2n!}  \tau^n,\\
  \bar \phi(1,\bar t+\tau) -\mu &\ge  + \frac\Lambda{2n!}  \tau^n.
\end{align*}
By uniform convergence of the adjoint state, there is $\nu_1\in (0,\nu_0)$ such that
\[
 \|\bar \phi(1,\cdot)-\phi_\nu(1,\cdot)\|_{C([\bar t-\tau,\bar t+\tau])} \le \frac\Lambda{4n!}  \tau^n
\]
for all $\nu<\nu_1$. This implies
\begin{align*}
  \phi_\nu(1,\bar t-\tau) -\mu &\le  - \frac\Lambda{4n!}  \tau^n < 0, \\
  \phi_\nu(1,\bar t+\tau) -\mu &\ge  + \frac\Lambda{4n!}  \tau^n > 0.
\end{align*}
By the intermediate value theorem, there is a solution to $\phi_\nu(1,t) -\mu = 0$ in $(\bar t-\tau,\bar t+\tau)$.
At least one of these solutions is indeed a switching point of $u_\nu$, as $\phi_\nu(1,t)-\mu$ changes sign in the
interval $(\bar t-\tau,\bar t+\tau)$.

Analogously, we can show existence of solutions of $\phi_\nu(1,t) -\mu+a\nu=0$.
Here, we obtain for $\nu<\nu_1$ making $\nu_1$ smaller if necessary
\begin{align*}
  \phi_\nu(1,\bar t-\tau) -\mu+a\nu &\le a\nu - \frac\Lambda{4n!}  \tau^n
  \le - \frac\Lambda{8n!}  \tau^n < 0, \\
  \phi_\nu(1,\bar t+\tau) -\mu+a\nu &\ge a\nu  + \frac\Lambda{4n!}  \tau^n
  \ge \frac\Lambda{8n!}  \tau^n> 0.
\end{align*}
This shows existence of solutions of \eqref{sw4} close to $\bar t$ for small $\nu$.

Let now $t\in (\bar t-\tau,\bar t + \tau)$ be a solution of \eqref{sw3}. Then we get
\[
 |\bar \phi(1,t) -  \phi_\nu(1,t) |= \left|\mu + \frac1{n!} \frac{\d^n}{\dt^n} \bar\varphi(1,\xi) (t-\bar t)^n - \mu\right| \ge \frac\Lambda{2n!} |t-\bar t|^n.
\]
Since $|\bar \phi(1,t) -  \phi_\nu(1,t)| \le c \nu^{1/2}$ by Lemma \ref{lem8}, we obtain $|t-\bar t| \le c' \ \nu^{\frac1{2n}}$.
If $t$ is a solution of  \eqref{sw4} in $(\bar t-\tau,\bar t + \tau)$, then we
obtain
\[
  |\bar \phi(1,t) - \phi_\nu(1,t) | = \left|\mu + \frac1{n!} \frac{\d^n}{\dt^n} \bar\varphi(1,\xi) (t-\bar t)^n - \mu+\nu a\right| \ge \frac\Lambda{2n!} |t-\bar t|^n - \nu|a|.
\]
This proves the estimate $|t-\bar t| \le c' \ \nu^{\frac1{2n}}$ for all $\nu<\nu_0$.
The inequality \eqref{eq44} can be proven by using Lemma \ref{lem6} instead of Lemma \ref{lem8} to estimate $|\bar \phi(1,t) -  \phi_\nu(1,t)|$.

With obvious modifications, the result can be proven if $\Lambda<0$ or $t=t_i^{-\mu}$ holds.
\end{proof}

\begin{remark}
\begin{enumerate}
 \item
A finite integer $n$ satisfying \eqref{E:slope_condition} must exist, since
the mapping $t \mapsto \bar \varphi(1,t) $ is analytic in $(-\infty,T)$.

 \item For $n=1$, we immediately obtain the following particular case of
 Theorem \ref{T9}: Assume that $t_j^\mu$ is a switching point of $\bar u$ such that
  \eqref{E:slope_condition} is satisfied with $n = 1$. Then  for all $0 < \nu < \nu_1$ exactly two switching
  points  $t_j^{\mu,\nu}$ and $t_j^{\mu-\nu a}$ of $u_\nu$ exist that solve equation \eqref{sw3} and \eqref{sw4},
  respectively. We have $\lim_{\nu \searrow 0} t_j^{\mu,\nu} = \lim_{\nu \searrow 0} t_j^{\mu-\nu a} = t_j^\mu$.
  An analogous result holds for $t_j^{-\mu}$ with points $t_j^{-\mu,\nu}$
  and  $t_j^{-\mu - \nu b}$.
\end{enumerate}
\end{remark}

\begin{corollary}\label{C1}
Under the same conditions as in the previous Theorem \ref{T9}, for
every point $\bar t\in (0,T)$ with $|\bar\phi(1,t)|=\mu$
there exists $c>0$ such that
\[
 \int_{\bar t-\tau}^{\bar t+\tau} \bar\phi(1,t) ( u(t) - \bar u(t)) + \mu \jmath'(\bar u(t);\, u(t) - \bar u(t))\dt
 \ge c \|u-\bar u\|_{L^1(\bar t-\tau,\bar t+\tau)}^{n+1}
\]
holds for all $u\in U_{ad}$, where $n$ and $\tau$ are as in Theorem \ref{T9}.
Here we used again the notation $\jmath(u):=|u|$.
\end{corollary}
\begin{proof}
Let $\bar t:=t_j^\mu$  be a switching point of $\bar u$.
As in the proof of Theorem \ref{T9}, cf., \eqref{eq043} and \eqref{eq45},
there is $\tau>0$ and $K=\frac{|\Lambda|}{2\cdot n!}>0$
such that for all $t\in (\bar t-\tau,\bar t+\tau)$ it holds $\bar\phi(1,t)\ge \frac\mu2$ and
\be\label{eq044}
 |\bar\phi(1,t)-\mu|\ge  K |t- \bar t |^n.
\ee
In addition, we have $\bar u(t) \in \{a,0\}$ for almost all $t$ in this interval.
Define for $\epsilon>0$
\[
A_\epsilon:=\{t\in (\bar t-\tau,\bar t+\tau): \  |\bar\phi(1,t)-\mu|\ge \epsilon\},
\quad I_\epsilon = (\bar t-\tau,\bar t+\tau) \setminus A_\epsilon .
\]
Hence, due to \eqref{eq044}, there is a constant $c'>0$ such that it holds $|I_\epsilon| \le c' \ \epsilon^{1/n}$ for all $\epsilon>0$.

Let now $u\in U_{ad}$ be given.
Take $t\in A_\epsilon$ with $\bar\phi(1,t)-\mu \ge \epsilon$. This implies $\bar u(t)=a < 0$ and
\[\begin{split}
 \bar\phi(1,t) ( u(t) - \bar u(t)) + \mu \jmath'(\bar u(t);\, u(t) - \bar u(t))
 &= (\bar\phi(1,t)-\mu) ( u(t) - \bar u(t)) \\
 & \ge \epsilon |u(t) - \bar u(t)|.
 \end{split}
\]
On the other hand,
take $t\in A_\epsilon$ with $\bar\phi(1,t)-\mu \le -\epsilon$.
As $\tau$ was chosen such that
$\bar\phi(1,t)\ge \frac\mu2>0$ holds on $(\bar t-\tau,\bar t+\tau)$,
it follows $\bar u(t)=0$ and
\[\begin{split}
 \bar\phi(1,t) ( u(t) - \bar u(t)) + \mu \jmath'(\bar u(t);\, u(t) - \bar u(t))
 &= \bar\phi(1,t) ( u(t) - \bar u(t)) + \mu  |u(t) - \bar u(t)| \\
 & \ge (\mu-\bar\phi(1,t)) |u(t) - \bar u(t)|\\
 & \ge \epsilon|u(t) - \bar u(t)|.
 \end{split}
\]
%where we used that $\bar\phi(1,t)\ge \frac\mu2>0$ holds on $(\bar t-\tau,\bar t+\tau)$.
Invoking \eqref{varineq_dir_pointw}, we conclude
\begin{multline*}
 \int_{\bar t-\tau}^{\bar t+\tau} \bar\phi(1,t) ( u(t) - \bar u(t)) + \mu \jmath'(\bar u(t);\, u(t) - \bar u(t))\dt\\
 \begin{aligned}
 &\ge \int_{A_\epsilon}\bar\phi(1,t) ( u(t) - \bar u(t)) + \mu \jmath'(\bar u(t);\, u(t) - \bar u(t))\dt\\
 &\ge \epsilon \|u-\bar u\|_{L^1(A_\epsilon)}\\
 &\ge \epsilon (\|u-\bar u\|_{L^1(\bar t-\tau,\bar t+\tau)} - \|u-\bar u\|_{L^1(I_\epsilon)})\\
 & \ge \epsilon(\|u-\bar u\|_{L^1(\bar t-\tau,\bar t+\tau)} - c' |b-a|\ \epsilon^{1/n}).
 \end{aligned}
\end{multline*}
Setting
\[
 \epsilon:= (2c' |b-a|)^{-n} \|u-\bar u\|_{L^1(\bar t-\tau,\bar t+\tau)}^n
\]
yields
\begin{multline*}
 \int_{\bar t-\tau}^{\bar t+\tau} \bar\phi(1,t) ( u(t) - \bar u(t)) + \mu \jmath'(\bar u(t);\, u(t) - \bar u(t))\dt
\ge \frac\epsilon2 \|u-\bar u\|_{L^1(\bar t-\tau,\bar t+\tau)}\\
= c \|u-\bar u\|_{L^1(\bar t-\tau,\bar t+\tau)}^{1+n}
\end{multline*}
with $c=(2c' |b-a|)^{-n}/2$.
\end{proof}
\begin{theorem}
Assume that $\|\bar y(\cdot,T) - y_\Omega\|_{L^2(0,1)} > 0$ is fulfilled.
Furthermore, we require that $|\bar\phi(1,T)|\ne \mu$.
Then there are constants $n\in \mathbb N$, $\nu_2>0$, and $c>0$ such that
it holds
\[
 \| u_\nu - \bar u\|_{L^1(0,T)} \le c\ \nu^{1/n}
\]
for all $\nu\in(0,\nu_2)$.
\end{theorem}
\begin{proof}
The condition $|\bar\phi(1,T)|\ne \mu$ implies that
there are finitely many solutions of $|\bar\phi(1,t)|=\mu$ in $(0,T)$.
Let $t_1 \dots t_m$ be the solutions of $|\bar\phi(1,t)|=\mu$ in $(0,T)$.
If $|\bar\phi(1,0)|=\mu$, then we add the point $t=0$ to this set.
Denote by $n_i$ and $\tau_i$ the constants
given by Theorem \ref{T9}
associated with the points $t_i$, $i=1\dots m$.
Note that Theorem \ref{T9} is also true if $t=0$ is taken, with obvious modifications of the proof.
Set $I_i:=(t_i -\tau_i,t_i+\tau_i)$,
$I = \bigcup_{i=1}^mI_i$,
$J:=(0,T)\setminus I$, $n:=\max_{i=1\dots m} n_i$.

By continuity, there is $\sigma>0$ such that $\big||\bar\phi(1,t)|-\mu\big|\ge \sigma$ for
all $t\in J$.
Due to uniform convergence $\phi_\nu \to \bar\phi$, there is $\nu_2>0$ such that it holds
$\big||\phi_\nu(1,t)|-\mu\big|\ge \sigma/2 > \nu(b-a)$
for all  $t\in J$ and $\nu\in(0,\nu_2)$.
Consequently $u_\nu$ and $\bar u$ coincide on $J$.

The intervals $I_i$ were constructed in Theorem \ref{T9}
such that $|\bar\phi(1,t)|>\mu/2$ holds for all $t\in I_i$.
Due to uniform convergence of the adjoint states, we get $|\phi_\nu(1,t)|>\mu/4$ for all $t\in I$ and for all $\nu\in(0,\nu_2)$
by making $\nu_2$ smaller if necessary.
Hence, the signs of $\bar\phi(1,t)$ and $\phi_\nu(1,t)$ coincide on $I$, which implies $\bar u(t)u_\nu(t)\ge 0$ for all $t\in I$.

In the next step we will invoke Corollary \ref{C1}. Let $c_i$ be the constant given by Corollary \ref{C1}
associated to the point $t_i$.
Then we find
\begin{multline*}
\int_0^T \bar\phi(1,t) ( u_\nu(t) - \bar u(t)) + \mu \jmath'(\bar u(t);\, u_\nu(t) - \bar u(t))\dt\\
\begin{aligned}
&= \sum_{i=1}^m \int_{I_i} \bar\phi(1,t) ( u_\nu(t) - \bar u(t)) + \mu \jmath'(\bar u(t);\, u_\nu(t) - \bar u(t))\dt\\
 &\ge \sum_{i=1}^m  c_i \|u_\nu-\bar u\|_{L^1(I_i)}^{1+n_i}.
 \end{aligned}
\end{multline*}
Since $|\bar u_\nu-u_\nu|\le b-a$ and $n_i\le n$, we obtain with $\tilde c:=\min\limits_{i=1\dots m}c_i(2\tau_i (b-a))^{n_i-n}$
\begin{multline}\label{eq48}
\int_0^T  \bar\phi(1,t) ( u_\nu(t) - \bar u(t)) + \mu \jmath'(\bar u(t);\, u_\nu(t) - \bar u(t))\dt\\
\begin{aligned}
& \ge \sum_{i=1}^m  c_i \|u_\nu-\bar u\|_{L^1(I_i)}^{1+n_i}\\
&  = \sum_{i=1}^m  c_i \|u_\nu-\bar u\|_{L^1(I_i)}^{n_i-n}\|u_\nu-\bar u\|_{L^1(I_i)}^{1+n}\\
& \ge \sum_{i=1}^m  c_i (2\tau_i (b-a))^{n_i-n} \|u_\nu-\bar u\|_{L^1(I_i)}^{1+n}\\
&\ge \tilde c  \|u_\nu-\bar u\|_{L^1(I)}^{n+1}.
\end{aligned}
\end{multline}
Testing the variational inequality \eqref{varineq_dir}  for $u_\nu$ with $\bar u$
and adding it to the inequality \eqref{eq48}, we obtain
\begin{multline*}
\int_0^T (\nu u_\nu(t)+\varphi_\nu(1,t) -\bar \phi(1,t))(\bar u(t) - u_\nu(t)) \dt\\
+   \mu ( j'(u_\nu; \ \bar u-u_\nu)  + j'(\bar u; \ u_\nu-\bar u))
\ge \tilde c  \|u_\nu-\bar u\|_{L^1(I)}^{n+1}.
\end{multline*}
Since $\bar u = u_\nu$ on $J$ and $\bar u(t)u_\nu(t)\ge 0$ on $I$
it holds $j'(u_\nu; \ \bar u-u_\nu)  + j'(\bar u; \ u_\nu-\bar u)=0$.
Due to the definition of the adjoint equation, it holds
\[
 \int_0^T (\varphi_\nu(1,t) -\bar \phi(1,t))(\bar u(t) - u_\nu(t)) \dt
 = - \|y_{u_\nu}(\cdot,T) -y_{\bar u}(\cdot,T) \|_{L^2(\Omega)}^2.
\]
Combining these facts, we find
\[
\|y_{u_\nu}(\cdot,T) -y_{\bar u}(\cdot,T) \|_{L^2(\Omega)}^2 + \nu \|u_\nu -\bar u\|_{L^2(0,T)}^2
 + \tilde c  \|u_\nu-\bar u\|_{L^1(I)}^{n+1}
\le \nu (\bar u,\bar u -u_\nu)_{L^2(I)}.
\]
See also \cite[Lemma 1.3]{Seydenschwanz2015} and \cite[Lemma 3.1]{WachsmuthWachsmuth2011a} for similar results.
By Young's inequality, we can estimate the right-hand side,
\begin{align*}
 \nu (\bar u,\bar u -u_\nu)_{L^2(I)} &=  \nu (\bar u,\bar u -u_\nu)_{L^2(I)}\\
 &\le \nu \,  |\max\{b,|a|\}| \, \|u_\nu-\bar u\|_{L^1(I)}\\
 &\le  \frac n{n+1}  \tilde c^{-1/n}(\nu |b-a| )^{\frac{n+1}n}  +  \frac{\tilde c}n  \|u_\nu-\bar u\|_{L^1(I)}^{n+1}.
\end{align*}
The last item can be absorbed by the left-hand side. This shows that the convergence rates
\begin{align*}
 \|y_{u_\nu}(\cdot,T) -y_{\bar u}(\cdot,T) \|_{L^2(\Omega)} &\le c \ \nu^{\frac12 + \frac1{2n}} ,\\
 \|u_\nu -\bar u\|_{L^2(0,T)} &\le c\ \nu^{\frac1{2n}},\\
 \|u_\nu-\bar u\|_{L^1(0,T)} &\le c \ \nu^{\frac1n}
\end{align*}
are satisfied for $\nu$ small enough
with some constant $c>0$ independent of $\nu$.
Here, we used again that $u_\nu=\bar u$ on $J$, which implies $\|u_\nu-\bar u\|_{L^1(I)}=\|u_\nu-\bar u\|_{L^1(0,T)}$.
\end{proof}

\begin{remark}
Let us point out some possible extensions of the previous theorem.
First, the assumption $|\bar\phi(1,T)|\ne \mu$ can be omitted
if we require instead that
there are finitely many solutions of $|\bar\phi(1,t)|=\mu$ in $(0,T)$,
and that there exists
$\tau>0$, $n>0$, and $c>0$ such that it holds
\[
 \Big|\, \{t\in (T-\tau,T):\ \big||\bar\phi(1,t)|-\mu\big|<\epsilon\} \Big| \le c\ \epsilon^{1/n} \quad \forall \epsilon>0.
\]
Then the conclusion of Corollary \ref{C1} is valid for the interval $(T-\tau,T)$,
and the proof of Theorem \ref{T9} remains valid with minor modifications.
\end{remark}

%%%%%%%%%%%
\section{Extensions}
%%%%%%%%%%%

The results of this paper can be easily extended to the following slightly more general situations:

(i) We considered problems with homogeneous initial condition $y(\cdot,0) = 0$. All results remain true for the non-homogeneous initial condition
$y(\cdot,0) = y_0(\cdot)$ with $y_0 \in L^2(\Omega)$. To see this, we solve the heat equation with homogeneous boundary data
and initial condition $y(\cdot,0) = y_0$ and denote the solution by $\hat y$. Then $y(x,T) - y_{\Omega} = y_u(x,t) -  (y_{\Omega} - \hat y(x,T))$
so that the results can be proven with $\hat y_\Omega :=  y_{\Omega} - \hat y(x,T)$.

(ii) For distributed controls of the form $f(x,t) = e(x)u(t)$ that act in the right-hand side of the heat equation with homogeneous boundary
conditions, the solution $y$ is given by the series representation
\begin{equation} \label{Eseries_yT_distr}
y(x,T) = \sum_{n=1}^\infty \frac{e_n}{N_n} \int_0^T e^{-\rho_n^2(T-s)}u(s)\,ds,
\end{equation}
where
\[
e_n := \int_0^1 \cos(\rho_n \xi) e(\xi) \, d\xi.
\]
In this way, the Fourier coefficients $e_n$ replace the numbers $\cos{\rho_n}$ in \eqref{Eseries_yT}. For proving the
switching properties,  in \eqref{cos_not_zero} we used the fact that $\cos{\rho_n} \not= 0$ holds for all $n \in \mathbb{N}$.
Therefore, an easy inspection of the proofs shows
that all results of the paper remain true for distributed controls of the form $f(x,t) = e(x)u(t)$ with fixed $e \in L^2(\Omega)$,
if the condition
\[
\int_0^1 \cos(\rho_n \xi) e(\xi) \, d\xi \not=0 \quad \forall n \in \mathbb{N}
\]
is fulfilled. In other words, the theory remains true for functions $e$ where  all Fourier coefficients with respect
to the system $\cos(\rho_n x)$ are non-vanishing.

\bibliographystyle{plain}
\bibliography{references}
\end{document}